\documentclass[12pt]{article}
\usepackage{amssymb}
\usepackage{amsfonts}
\usepackage{amsmath}
\usepackage{amsthm}
\newcommand{\be}{\begin{equation}}
\newcommand{\ee}{\end{equation}}
\newcommand{\bea}{\begin{eqnarray}}
\newcommand{\eea}{\end{eqnarray}}
\newcommand{\ba}{\begin{array}}
\newcommand{\ea}{\end{array}}

\newcommand{\bc}{\begin{center}}
\newcommand{\ec}{\end{center}}
\newcommand{\ben}{\begin{enumerate}}
\newcommand{\een}{\end{enumerate}}
\newcommand{\bfi}{\begin{figure}}
\newcommand{\efi}{\end{figure}}

\newcommand{\bq}{\begin{quote}}
\newcommand{\eq}{\end{quote}}
\newcommand{\bqu}{\begin{quotation}}
\newcommand{\equ}{\end{quotation}}
\newenvironment{emphit}{\begin{itemize}}{\end{itemize}}
\newcommand{\bemp}{\begin{emphit}}
\newcommand{\eemp}{\end{emphit}}

\newcommand{\bt}{\begin{tabular}}
\newcommand{\et}{\end{tabular}}

\newtheorem{myth}{Theorem}[section]
\newtheorem{mylem}{Lemma}[section]
\newtheorem{mycor}{Corollary}[section]

\newtheorem{mydef}{Definition}[section]

\begin{document}
\date{}
\title{New Identities In Universal Osborn Loops 
}
\author{T\`em\'it\d {\'o}p\d {\'e} Gb\d {\'o}l\'ah\`an Jaiy\'e\d
ol\'a$^1$\thanks{Corresponding author}\\
 \&\\
John Ol\'us\d ol\'a Ad\'en\'iran$^2$} \maketitle
\begin{abstract}
A question associated with the 2005 open problem of Michael Kinyon
(Is every Osborn loop universal?), is answered. Two nice identities
that characterize universal (left and right universal) Osborn loops
are established. Numerous new identities are established for
universal (left and right universal) Osborn loops like CC-loops,
VD-loops and universal weak inverse property loops. Particularly,
Moufang loops are discovered to obey the new identity
$[y(x^{-1}u)\cdot u^{-1}](xu)=[y(xu)\cdot u^{-1}](x^{-1}u)$
surprisingly. For the first time, new loop properties that are
weaker forms of well known loop properties like inverse property,
power associativity and diassociativity are introduced and studied
in universal (left and right universal) Osborn loops. Some of them
are found to be necessary and sufficient conditions for a universal
Osborn to be 3 power associative. For instance, four of them are
found to be new necessary and sufficient conditions for a CC-loop to
be power associative. A conjugacy closed loop is shown to be
diassociative if and only if it is power associative and has a weak
form of diassociativity.
\end{abstract}

\section{Introduction}
\paragraph{}
The isotopic invariance of varieties of quasigroups and loops
described by one or more equivalent identities, especially those
that fall in the class of Bol-Moufang type loops have been of
interest to researchers in loop theory in the recent past. These
types of identities were first named by Fenyves \cite{phd56} and
\cite{phd50} in the 1960s and later on in this $21^\textrm{st}$
century by Phillips and Vojt\v echovsk\'y \cite{phd9}, \cite{phd61}
and \cite{phd124}. Among such are Etta Falconer \cite{phd159} and
\cite{phd160} which investigated isotopy invariants in quasigroups.
Loops such as Bol loops, Moufang loops, central loops and extra
loops are the most popular loops of Bol-Moufang type whose isotopic
invariance have been considered. For an overview of the theory of
loops, readers may check \cite{phd3,phd41,phd39,phd49,phd42,phd75}.

Consider $(G,\cdot )$ and $(H,\circ )$ been two distinct groupoids
(quasigroups, loops). Let $A,B$ and $C$ be three bijective mappings,
that map $G$ onto $H$. The triple $\alpha =(A,B,C)$ is called an
\textit{isotopism} of $(G,\cdot )$ onto $(H,\circ )$ if and only if
$$xA\circ yB=(x\cdot y)C~\forall~x,y\in G.$$
So, $(H,\circ )$ is called a groupoid (quasigroup, loop)
\textit{isotope} of $(G,\cdot )$.

If $C=I$ is the identity map on $G$ so that $H=G$, then the triple
$\alpha =(A,B,I)$ is called a \textit{principal isotopism} of
$(G,\cdot )$ onto $(G,\circ )$ and $(G,\circ)$ is called a
\textit{principal isotope} of $(G,\cdot )$. Eventually, the equation
of relationship now becomes
$$x\cdot y=xA\circ yB~\forall~x,y\in G$$
which is easier to work with. But if $A=R_g$ and $B=L_f$ where
$R_x~:G\to G$, the \textit{right translation} is defined by
$yR_x=y\cdot x$ and $L_x~:G\to G$, the \textit{left translation} is
defined by $yL_x=x\cdot y$ for all $x,y\in G$, for some $f,g\in G$,
the relationship now becomes
$$x\cdot y=xR_g\circ yL_f~\forall~x,y\in G$$
or
$$x\circ y=xR_g^{-1}\cdot yL_f^{-1}~\forall~x,y\in G.$$
With this new form, the triple $\alpha =(R_g,L_f,I)$ is called an
\textit{$f,g$-principal isotopism} of $(G,\cdot )$ onto $(G,\circ
)$, $f$ and $g$ are called \textit{translation elements} of $G$ or
at times written in the pair form $(g,f)$, while $(G,\circ )$ is
called an \textit{$f,g$-principal isotope} of $(G,\cdot )$.

The last form of $\alpha$ above given rises to an important result
in the study of loop isotopes of loops.

\begin{myth}\label{1:1}(Bruck \cite{phd41})

Let $(G,\cdot )$ and $(H,\circ )$ be two distinct isotopic loops.
For some $f,g\in G$, there exists an $f,g$-principal isotope
$(G,\ast )$ of $ (G,\cdot )$ such that $(H,\circ )\cong (G,\ast )$.
\hspace{5.00cm}$\Box$
\end{myth}

With this result, to investigate the isotopic invariance of an
isomorphic invariant property in loops, one simply needs only to
check if the property in consideration is true in all
$f,g$-principal isotopes of the loop. A property is \textit{isotopic
invariant} if whenever it holds in the domain loop i.e. $(G,\cdot )$
then it must hold in the co-domain loop i.e. $(H,\circ )$ which is
an isotope of the formal. In such a situation, the property in
consideration is said to be a \textit{universal property} hence the
loop is called a \textit{universal loop} relative to the property in
consideration as often used by Nagy and Strambach \cite{phd88} in
their algebraic and geometric study of the universality of some
types of loops. For instance, if every isotope of a "certain" loop
is a "certain" loop, then the formal is called a \textit{universal
"certain" loop}. So, we can now restate Theorem~\ref{1:1} as :

\begin{myth}\label{1:2}
Let $(G,\cdot )$ be a "certain" loop where "certain" is an
isomorphic invariant property. $(G,\cdot )$ is a universal "certain"
loop if and only if every $f,g$-principal isotope $(G,\ast )$ of $
(G,\cdot )$ has the "certain" loop property.\hspace{9.00cm}$\Box$
\end{myth}

\paragraph{}
From the earlier discussions, if $(H,\circ )=(G,\cdot )$ then the
triple $\alpha =(A,B,C)$ is called an \textit{autotopism} where
$A,B,C\in SYM(G,\cdot )$, the set of all bijections on $(G,\cdot )$
called the \textit{symmetric group} of $(G,\cdot )$. Such triples
form a group $AUT(G,\cdot )$ called the \textit{autotopism group} of
$(G,\cdot )$.

\paragraph{}
Bol-Moufang type of quasigroups (loops) are not the only quasigroups
(loops) that are isomorphic invariant and whose universality have
been considered. Some others are weak inverse property loops (WIPLs)
and cross inverse property loops (CIPLs). The universality of WIPLs
and CIPLs have been addressed by Osborn \cite{phd89} and Artzy
\cite{phd30} respectively. In 1970, Basarab \cite{phd146} later
continued the work of Osborn of 1961 on universal WIPLs by studying
isotopes of WIPLs that are also WIPLs after he had studied a class
of WIPLs (\cite{phd149}) in 1967. Osborn \cite{phd89}, while
investigating the universality of WIPLs discovered that a universal
WIPL $(G,\cdot )$ obeys the identity
\begin{equation}\label{eq:1}
yx\cdot (zE_y\cdot y)=(y\cdot xz)\cdot y~\forall~x,y,z\in G
\end{equation}
where
$E_y=L_yL_{y^\lambda}=R_{y^\rho}^{-1}R_y^{-1}=L_yR_yL_y^{-1}R_y^{-1}$
and $y^\lambda$ and $y^\rho$ are respectively the left and right
inverse elements of $y$.

\paragraph{}
Eight years after Osborn's \cite{phd89} 1960 work on WIPL, in 1968,
Huthnance Jr. \cite{phd44} studied the theory of generalized Moufang
loops. He named a loop that obeys (\ref{eq:1}) a \textit{generalized
Moufang loop} and later on in the same thesis, he called them
\textit{M-loops}. On the other hand, he called a \textit{universal
WIPL} an \textit{Osborn loop} and this same definition was adopted
by Chiboka \cite{phd96}. Basarab dubbed a loop $(G,\cdot )$
satisfying the identity:
\begin{equation}\label{eq:2}
x(yz\cdot x)=(x\cdot yE_x)\cdot zx~\forall~x,y,z\in G
\end{equation}
an \textit{Osborn loop} where
$E_x=R_xR_{x^\rho}=(L_xL_{x^\lambda})^{-1}=R_xL_xR_x^{-1}L_x^{-1}$.

It will look confusing if both Basarab's and Huthnance's definitions
of an Osborn loop are both adopted because an Osborn loop of Basarab
is not necessarily a universal WIPL (Osborn loop of Huthnance). So
in this work, Huthnance's definition of an Osborn loop will be
dropped while we shall stick to that of Basarab which was actually
adopted by Kinyon \cite{phd33} and the open problem we intend to
solve is relative to Basarab's definition of an Osborn loop and not
that of Huthnance. Huthnance \cite{phd44} was able to deduce some
properties of $E_x$ relative to (\ref{eq:1}).
$E_x=E_{x^\lambda}=E_{x^\rho}$. So, since $E_x=R_xR_{x^\rho}$, then
$E_x=E_{x^\lambda}=R_{x^\lambda}R_{x}$ and
$E_x=(L_{x^\rho}L_x)^{-1}$. So, we now have two identities
equivalent to identities (\ref{eq:1}) and (\ref{eq:2}) defining an
Osborn loop.

\begin{equation}\label{eq:1.1}
\textrm{OS$_0$}~:~x(yz\cdot x)=x(yx^\lambda \cdot x)\cdot zx
\end{equation}
\begin{equation}\label{eq:1.2}
\textrm{OS$_1$}~:~x(yz\cdot x)=x(yx\cdot x^\rho )\cdot zx
\end{equation}
Although Basarab \cite{phd148} and \cite{phd170} considered
universal Osborn loops but the universality of Osborn loops was
raised as an open problem by Michael Kinyon in 2005 at a conference
tagged "Milehigh Conference on Loops, Quasigroups and
Non-associative Systems" held at the University of Denver, where he
presented a talk titled "A Survey of Osborn Loops". The present
authors have been able to find a counter example to prove that not
every Osborn loop is universal (in a different paper, submitted for
publication) thereby putting the open problem to rest. Kinyon
\cite{phd33} further raised the question concerning the problem by
asking if there exists a 'nice' identity that characterizes a
universal Osborn loop.

\paragraph{}
In this study, a question associated with the 2005 open problem of
Michael Kinyon (Is every Osborn loop universal?), is answered. Two
nice identities that characterize universal (left and right
universal) Osborn loops are established. Numerous new identities are
established for universal Osborn loops like CC-loops, VD-loops and
universal weak inverse property loops. Particularly, Moufang loops
are discovered to obey the new identity $[y(x^{-1}u)\cdot
u^{-1}](xu)=[y(xu)\cdot u^{-1}](x^{-1}u)$ surprisingly. For the
first time, new loop properties that are weaker forms of well known
loop properties like inverse property, power associativity and
diassociativity are introduced and studied in universal (left and
right universal) Osborn loops. Some of them are found to be
necessary and sufficient conditions for a universal Osborn to be 3
power associative. For instance, four of them are found to be new
necessary and sufficient conditions for a CC-loop to be power
associative. A conjugacy closed loop is shown to be diassociative if
and only if it is power associative and has a weak form of
diassociativity.

\section{Preliminaries}
Let $G$ be a non-empty set. Define a binary operation ($\cdot $) on
$G$.

If $x\cdot y\in G$ for all $x, y\in G$, then the pair $(G, \cdot )$
is called a \textit{groupoid} or \textit{Magma}.

If each of the equations:
\begin{displaymath}
a\cdot x=b\qquad\textrm{and}\qquad y\cdot a=b
\end{displaymath}
has unique solutions in $G$ for $x$ and $y$ respectively, then $(G,
\cdot )$ is called a \textit{quasigroup}.

A quasigroup is therefore an \textit{algebra} having a binary
multiplication $x\cdot y$ usually written $xy$ which satisfies the
conditions that for any $a,b$ in the quasigroup the equations
\begin{displaymath}
a\cdot x=b\qquad\textrm{and}\qquad y\cdot a=b
\end{displaymath}
have unique solutions for $x$ and $y$ lying in the quasigroup.

If there exists a unique element $e\in G$ called the
\textit{identity element} such that for all $x\in G$, $x\cdot
e=e\cdot x=x$, $(G, \cdot )$ is called a \textit{loop}. We write
$xy$ instead of $x\cdot y$, and stipulate that $\cdot$ has lower
priority than juxtaposition among factors to be multiplied. For
instance, $x\cdot yz$ stands for $x(yz)$.

It can now be seen that a groupoid $(G, \cdot )$ is a quasigroup if
it's left and right translation mappings are bijections or
permutations. Since the left and right translation mappings of a
loop are bijective, then the inverse mappings $L_x^{-1}$ and
$R_x^{-1}$ exist. Let
\begin{displaymath}
x\backslash y =yL_x^{-1}=y\mathbb{L}_x\qquad\textrm{and}\qquad
x/y=xR_y^{-1}=x\mathbb{R}_y
\end{displaymath}
and note that
\begin{displaymath}
x\backslash y =z\Longleftrightarrow x\cdot
z=y\qquad\textrm{and}\qquad x/y=z\Longleftrightarrow z\cdot y=x.
\end{displaymath}
Hence, $(G, \backslash )$ and $(G, /)$ are also quasigroups. Using
the operations ($\backslash$) and ($/$), the definition of a loop
can be stated as follows.

\begin{mydef}\label{0:1}\textrm{
A \textit{loop} $(G,\cdot ,/,\backslash ,e)$ is a set $G$ together
with three binary operations ($\cdot $), ($/$), ($\backslash$) and
one nullary operation $e$ such that
\begin{description}
\item[(i)] $x\cdot (x\backslash y)=y$, $(y/x)\cdot x=y$ for all
$x,y\in G$,
\item[(ii)] $x\backslash (x\cdot y)=y$, $(y\cdot x)/x=y$ for all
$x,y\in G$ and
\item[(iii)] $x\backslash x=y/y$ or $e\cdot x=x$ for all
$x,y\in G$.
\end{description}}
\end{mydef}
We also stipulate that ($/$) and ($\backslash$) have higher priority
than ($\cdot $) among factors to be multiplied. For instance,
$x\cdot y/z$ and $x\cdot y\backslash z$ stand for $x(y/z)$ and
$x\cdot (y\backslash z)$ respectively.

In a loop $(G,\cdot )$ with identity element $e$, the \textit{left
inverse element} of $x\in G$ is the element $xJ_\lambda
=x^\lambda\in G$ such that
\begin{displaymath}
x^\lambda\cdot x=e
\end{displaymath}
while the \textit{right inverse element} of $x\in G$ is the element
$xJ_\rho =x^\rho\in G$ such that
\begin{displaymath}
x\cdot x^\rho=e.
\end{displaymath}

\begin{mydef} A loop $(Q,\cdot )$ is called:
\begin{description}
\item[(a)] a \textit{3 power associative property loop (3-PAPL)} if and only if $xx\cdot
x=x\cdot xx$ for all $x\in Q$.
\item[(b)] a left self inverse property loop (LSIPL) if and only if
$x^\lambda\cdot xx=x$ for all $x\in Q$.
\item[(c)] a right self inverse property loop (RSIPL) if and only if
$xx\cdot x^\rho =x$ for all $x\in Q$.
\item[(d)] a self automorphic inverse property loop (SFAIPL) if and only if
$(xx)^\rho =x^\rho x^\rho$ for all $x\in Q$.
\item[(e)] a self weak inverse property loop (SWIPL) if and only if
$x\cdot (xx)^\rho =x^\rho$ for all $x\in Q$.
\item[(f)] a left ${}^1$bi-self inverse property loop (L${}^1$BSIPL) if and only if
$x^\lambda (xx\cdot x)=xx$ for all $x\in Q$.
\item[(g)] a left ${}^2$bi-self inverse property loop (L${}^2$BSIPL) if and only if
$x^\lambda (x\cdot xx)=xx$ for all $x\in Q$.
\end{description}
\end{mydef}

\begin{mydef}
Let $(Q,\cdot )$ be a loop and let $w_1(q_1,q_2,\cdots ,q_n)$ and
$w_2(q_1,q_2,\cdots ,q_n)$ be words in terms of variables
$q_1,q_2,\cdots ,q_n$ of the loop $Q$ with equal lengths $N$($N\in
\mathbb{N}$,~$N>1$) such that the variables $q_1,q_2,\cdots ,q_n$
appear in them in equal number of times. $Q$ is called a {\Large
N}$_{w_1(r_1,r_2,\cdots , r_n)=w_2(r_1,r_2,\cdots
,r_n)}^{m_1,m_2,\cdots ,m_n}$ loop if it obeys the identity
$w_1(q_1,q_2,\cdots ,q_n)=w_2(q_1,q_2,\cdots ,q_n)$ where
$m_1,m_2,\cdots ,m_n\in \mathbb{N}$ represent the number of times
the variables $q_1,q_2,\cdots ,q_n\in Q$ respectively appear in the
word $w_1$ or $w_2$ such that the mappings $q_1\mapsto
r_1,q_1\mapsto r_1,\cdots ,q_n\mapsto r_n$ are assumed,
$r_1,r_2,\cdots r_n\in \mathbb{N}$.
\end{mydef}
In this study, we shall concentrate on when $N=4$.

The identities describing the most popularly known varieties of
Osborn loops are given below.
\begin{mydef}
A loop $(Q,\cdot )$ is called:
\begin{description}
\item[(a)](Basarab \cite{phd137}) a VD-loop if and only if $(\cdot )_x=(\cdot )^{L_x^{-1}R_x}$ and ${}_x(\cdot
)=(\cdot )^{R_x^{-1}L_x}$ i.e. $R_x^{-1}L_x\in PS_\lambda (Q,\cdot
)$ with companion $c=x$ and $L_x^{-1}R_x\in PS_\rho (Q,\cdot )$ with
companion $c=x$ for all $x\in Q$ where $PS_\lambda (Q,\cdot )$ and
$PS_\rho (Q,\cdot )$ are respectively the left and right
pseudo-automorphism groups of $Q$.
\item[(b)] a Moufang loop if and only if the identity $(xy)\cdot (zx)=(x\cdot
yz)x$ holds in $Q$.
\item[(c)] a conjugacy closed loop (CC-loop) if and only if the
identities $x\cdot yz=(xy)/x\cdot xz$ and $zy\cdot x=zx\cdot
x\backslash(yx)$ hold in $Q$.
\item[(d)] a universal WIPL if and only if the identity $x(yx)^\rho=y^\rho$ or $(xy)^\lambda
x=y^\lambda$ holds in $Q$ and all its isotopes.
\end{description}
\end{mydef}
All these four varieties of Osborn loops are universal. CC-loops,
and VD-loops are G-loops. \textit{G-loops} are loops that are
isomorphic to all their loop isotopes. Kunen \cite{phd185} has
studied them.

\begin{mydef}
Let the triple $\alpha =(A,B,C)$ be an isotopism of the groupoid
$(G,\cdot )$ onto a groupoid $(H,\circ )$.
\begin{description}
\item[(a)] If $\alpha =(A,B,B)$, then the triple is
called a left isotopism and the groupoids are called left isotopes.
\item[(b)] If $\alpha =(A,B,A)$, then the triple is
called a right isotopism and the groupoids are called right
isotopes.
\item[(c)] If $\alpha =(A,I,I)$, then the triple is
called a left principal isotopism and the groupoids are called left
principal isotopes.
\item[(d)] If $\alpha =(I,B,I)$, then the triple is
called a right principal isotopism and the groupoids are called
right principal isotopes.
\end{description}
\end{mydef}
A loop is a \textit{left (right) universal "certain" loop} if and
only if all its left (right) isotopes are "certain" loops.
\begin{myth}\label{1:1.1}
Let $(G,\cdot )$ and $(H,\circ )$ be two distinct left (right)
isotopic loops with the former having an identity element $e$. For
some $g~(f)\in G$, there exists an $e,g$~($f,e$)-principal isotope
$(G,\ast )$ of $ (G,\cdot )$ such that $(H,\circ )\cong (G,\ast )$.
\end{myth}
\begin{proof}
The proof of this is similar to that of Theorem~III.2.1 of
\cite{phd3}.
\end{proof}

\begin{myth}\label{1:2.0.1}
Let $(G,\cdot )$ be a "certain" loop where "certain" is an
isomorphic invariant property. $(G,\cdot )$ is a left (right)
universal "certain" loop if and only if every left (right) principal
isotope $(G,\ast )$ of $ (G,\cdot )$ has the "certain" loop
property.
\end{myth}
\begin{proof}
Use Theorem~\ref{1:1.1}.\end{proof}

\section{Main Results}
\begin{myth}\label{1:4}
A loop $(Q, \cdot ,\backslash ,/)$ is a universal Osborn loop if and
only if it obeys the identity{\footnotesize
\begin{displaymath}
\underbrace{ x\cdot u\backslash\{(yz)/v\cdot [u\backslash
(xv)]\}=(x\cdot u\backslash\{[y(u\backslash([(uv)/(u\backslash
(xv))]v))]/v\cdot [u\backslash (xv)]\})/v  \cdot
u\backslash[((uz)/v)(u\backslash (xv))]}_{\textrm{OS$_0'$}}
\end{displaymath}}
\begin{displaymath}
\textrm{or}
\end{displaymath}
\begin{displaymath}
\underbrace{ x\cdot u\backslash\{(yz)/v\cdot [u\backslash
(xv)]\}=\{x\cdot u\backslash\{[y(u\backslash (xv))]/v\cdot
[x\backslash (uv)]\}\}/v \cdot u\backslash[((uz)/v)(u\backslash
(xv))].}_{\textrm{OS$_1'$}}
\end{displaymath}
\end{myth}
\begin{proof}Let $\mathcal{Q}=(Q, \cdot ,\backslash ,/)$ be an
Osborn loop with any arbitrary principal isotope $\mathfrak{Q}=(Q,
\vartriangle,\nwarrow ,\nearrow )$ such that
\begin{displaymath}
x\vartriangle y=xR_v^{-1}\cdot yL_u^{-1}=(x/v)\cdot (u\backslash
y)~\forall~u,v\in Q.
\end{displaymath}
If $\mathcal{Q}$ is a universal Osborn loop then, $\mathfrak{Q}$ is
an Osborn loop. $\mathfrak{Q}$ obeys identity OS$_0$ implies
\begin{equation}\label{eq:3}
x\vartriangle [(y\vartriangle z)\vartriangle x]=\{x\vartriangle
[(y\vartriangle x^{\lambda '})\vartriangle x]\}\vartriangle
(z\vartriangle x)
\end{equation}
where $x^{\lambda '}=xJ_{\lambda '}$ is the left inverse of $x$ in
$\mathfrak{Q}$. The identity element of the loop $\mathfrak{Q}$ is
$uv$. So,
\begin{displaymath}
x\vartriangle y=xR_v^{-1}\cdot yL_u^{-1}~\textrm{implies}~
y^{\lambda '}\vartriangle y=y^{\lambda '}R_v^{-1}\cdot
yL_u^{-1}=uv~\textrm{implies}
\end{displaymath}
\begin{displaymath}
y^{\lambda '}R_v^{-1}R_{yL_u^{-1}}=uv~\textrm{implies}~yJ_{\lambda
'}=(uv)R_{yL_u^{-1}}^{-1}R_v=(uv)R_{(u\backslash
y)}^{-1}R_v=[(uv)/(u\backslash y)]v.
\end{displaymath}
Thus, using the fact that
\begin{displaymath}
x\vartriangle y=(x/v)\cdot (u\backslash y),
\end{displaymath}
$\mathfrak{Q}$ is an Osborn loop if and only if{\footnotesize
\begin{displaymath}
(x/v)\cdot u\backslash\{[(y/v)\cdot (u\backslash z)]/v\cdot
(u\backslash x)\}=((x/v)\cdot
u\backslash\{[(y/v)(u\backslash([(uv)/(u\backslash x)]v))]/v\cdot
(u\backslash x)\})/v\cdot u\backslash[(z/v)(u\backslash x)].
\end{displaymath}}
Do the following replacements:
\begin{displaymath}
x'=x/v\Rightarrow x=x'v,~z'=u\backslash z\Rightarrow
z=uz',~y'=y/v\Rightarrow y=y'v
\end{displaymath}
we have{\footnotesize
\begin{displaymath}
x'\cdot u\backslash\{(y'z')/v\cdot [u\backslash (x'v)]\}=(x'\cdot
u\backslash\{[y'(u\backslash([(uv)/(u\backslash (x'v))]v))]/v\cdot
[u\backslash (x'v)]\})/v  \cdot u\backslash[((uz')/v)(u\backslash
(x'v))].
\end{displaymath}}
This is precisely identity OS$_0'$ by replacing $x'$, $y'$ and $z'$
by $x$, $y$ and $z$ respectively.

The proof of the converse is as follows. Let $\mathcal{Q}=(Q, \cdot
,\backslash ,/)$ be an Osborn loop that obeys identity OS$_0'$.
Doing the reverse process of the proof of the necessary part, it
will be observed that equation (\ref{eq:3}) is true for any
arbitrary $u,v$-principal isotope $\mathfrak{Q}=(Q, \vartriangle
,\nwarrow ,\nearrow )$ of $\mathcal{Q}$. So, every $f,g$-principal
isotope $\mathfrak{Q}$ of $\mathcal{Q}$ is an Osborn loop. Following
Theorem~\ref{1:2}, $\mathcal{Q}$ is a universal Osborn loop if and
only if $\mathfrak{Q}$ is an Osborn loop.

The proof for the second identity is done similarly by using
identity OS$_1$.\end{proof}

\begin{mylem}\label{1:8}
Let $Q$ be a loop with multiplication group
$\mathcal{M}\textrm{ult}(Q)$. $Q$ is a universal Osborn loop if and
only if the triple $\big(\alpha (x,u,v),\beta (x,u,v),\gamma
(x,u,v)\big)\in AUT(Q)$ or the triple $\Big(R_{[u\backslash
(xv)]}\mathbb{R}_vR_{[x\backslash (uv)]}\mathbb{R}_{[u\backslash
(xv)]}R_v\gamma (x,u,v)\mathbb{R}_v,\beta (x,u,v),\gamma
(x,u,v)\Big)\in AUT(Q)$ for all $x,u,v\in Q$ where $\alpha
(x,u,v)=R_{_{(u\backslash([(uv)/(u\backslash
(xv))]v))}}\mathbb{R}_vR_{[u\backslash
(xv)]}\mathbb{L}_uL_x\mathbb{R}_v,~\beta
(x,u,v)=L_u\mathbb{R}_vR_{[u\backslash (xv)]}\mathbb{L}_u$ and
$\gamma (x,u,v)=\mathbb{R}_vR_{[u\backslash (xv)]}\mathbb{L}_uL_x$
are elements of $\mathcal{M}\textrm{ult}(Q)$.
\end{mylem}
\begin{proof}
This is obtained from identity OS$_0'$ or OS$_1'$ of
Theorem~\ref{1:4}.\end{proof}

\begin{myth}\label{1:9}
Let $Q$ be a loop with multiplication group
$\mathcal{M}\textrm{ult}(Q)$. If $Q$ is a universal Osborn loop,
then the triple $\big(\gamma
(x,u,v)\mathbb{R}_{_{(u\backslash[(u/v)(u\backslash (xv))])}},\beta
(x,u,v),\gamma (x,u,v)\big)\in AUT(Q)$ for all $x,u,v\in Q$ where
$\beta (x,u,v)=L_u\mathbb{R}_vR_{[u\backslash (xv)]}\mathbb{L}_u$
and $\gamma (x,u,v)=\mathbb{R}_vR_{[u\backslash
(xv)]}\mathbb{L}_uL_x$ are elements of $\mathcal{M}\textrm{ult}(Q)$.
\end{myth}
\begin{proof}
Theorem~\ref{1:4} will be employed. Let $z=e$ in identity OS$_0'$,
then {\footnotesize
\begin{displaymath}
x\cdot u\backslash\{y/v\cdot [u\backslash (xv)]\}=(x\cdot
u\backslash\{[y(u\backslash([(uv)/(u\backslash (xv))]v))]/v\cdot
[u\backslash (xv)]\})/v  \cdot u\backslash[(u/v)(u\backslash (xv))].
\end{displaymath}}
So, identity OS$_0'$ can now be written as{\footnotesize
\begin{displaymath}
x\cdot u\backslash\{(yz)/v\cdot [u\backslash (xv)]\}=\Big\{\{x\cdot
u\backslash[y/v\cdot (u\backslash
(xv))]\}/\{u\backslash[(u/v)(u\backslash (xv))]\}\Big\}\cdot
u\backslash[((uz)/v)(u\backslash (xv))].
\end{displaymath}}
From where we obtain $\Big(\gamma
(x,u,v)\mathbb{R}_{_{(u\backslash[(u/v)(u\backslash (xv))])}},\beta
(x,u,v),\gamma (x,u,v)\Big)\in AUT(Q)$.\end{proof}

\begin{mylem}\label{1:9a}
Let $(Q, \cdot ,\backslash ,/)$ be a universal Osborn loop. The
following identities are satisfied:
\begin{gather*}
\underbrace{y\{u\backslash ([(uv)/(u\backslash
(xv))]v)\}=\{(y[u\backslash (xv)])/v\cdot [x\backslash
(uv)]\}/[u\backslash (xv)]\cdot v}_{\textrm{OSI$_{01}$}},\\
\underbrace{\{(uz)/v\cdot u\backslash (\{(yv)(u\backslash
([(uv)/z]v))\}/v\cdot z)\}/v\cdot (u\backslash [(u/v)z])=(uz)/v\cdot
u\backslash (yz)}_{\textrm{OSI$_{01.1}$}}~\textrm{and}\\
\underbrace{(uz)/v\cdot u\backslash \{(yv\cdot z)/v\cdot
[((uz)/v)\backslash (uv)]\}=[(uz)/v\cdot u\backslash
(yz)]/\{u\backslash [(u/v)z]\}\cdot v}_{\textrm{OSI$_{01.2}$}}.\\
\textrm{Furthermore},~\underbrace{\{u\backslash(\{(uy\cdot
u)(u\backslash(uu\cdot u))\}/u)\}/u\cdot
u^\rho=y}_{\textrm{OSI$_{01.1.1}$}},\qquad uu\cdot u\backslash
(uu\cdot u)=(u\cdot uu)u,\\{\small \underbrace{v^\lambda\cdot
u\backslash \{(yv\cdot u^\rho)/v\cdot [v^\lambda\backslash
(uv)]\}=[v^\lambda\cdot u\backslash (yu^\rho)]/\{u\backslash
[(u/v)u^\rho]\}\cdot
v}_{\textrm{OSI$_{01.2.1}$}}},~\underbrace{v^\lambda (y\cdot
v^\lambda\backslash v)=(v^\lambda y)/v^\lambda\cdot
v}_{\textrm{OSI$_{01.2.2}$}},\\
v^\lambda\cdot (v\cdot v^\lambda\backslash v)=v^{\lambda^2}\cdot
v=(v^\lambda\cdot vv)v~\textrm{and}~v(v^\rho\cdot v\backslash v^\rho
)=v^\lambda\cdot v^\rho
\end{gather*}
are also satisfied.
\end{mylem}
\begin{proof}
To prove these identities, we shall make use of the three
autotopisms in Lemma~\ref{1:8} and Theorem~\ref{1:9}. In a
quasigroup, any two components of an autotopism uniquely determine
the third. So equating the first components of the three
autotopisms, it is easy to see that
\begin{displaymath}
\alpha (x,u,v)=\gamma
(x,u,v)\mathbb{R}_{_{(u\backslash[(u/v)(u\backslash
(xv))])}}=R_{[u\backslash (xv)]}\mathbb{R}_vR_{[x\backslash
(uv)]}\mathbb{R}_{[u\backslash (xv)]}R_v\gamma (x,u,v)\mathbb{R}_v.
\end{displaymath}
The establishment of the identities OSI$_{01}$, OSI$_{01.1}$ and
OSI$_{01.2}$ follows by using the bijections appropriately to map an
arbitrary element $y\in Q$ as follows:
\begin{description}
\item[OSI$_{01}$]
\begin{displaymath}
\alpha (x,u,v)=R_{[u\backslash (xv)]}\mathbb{R}_vR_{[x\backslash
(uv)]}\mathbb{R}_{[u\backslash (xv)]}R_v\gamma
(x,u,v)\mathbb{R}_v~\textrm{implies that}
\end{displaymath}
\begin{displaymath}
R_{_{(u\backslash([(uv)/(u\backslash
(xv))]v))}}\mathbb{R}_vR_{[u\backslash
(xv)]}\mathbb{L}_uL_x\mathbb{R}_v=
\end{displaymath}
\begin{displaymath}
R_{_{(u\backslash([(uv)/(u\backslash (xv))]v))}}\gamma
(x,u,v)\mathbb{R}_v=R_{[u\backslash
(xv)]}\mathbb{R}_vR_{[x\backslash (uv)]}\mathbb{R}_{[u\backslash
(xv)]}R_v\gamma (x,u,v)\mathbb{R}_v~\textrm{which gives}
\end{displaymath}
\begin{displaymath}
R_{_{(u\backslash([(uv)/(u\backslash (xv))]v))}}=R_{[u\backslash
(xv)]}\mathbb{R}_vR_{[x\backslash (uv)]}\mathbb{R}_{[u\backslash
(xv)]}R_v.
\end{displaymath}
So, for any $y\in Q$,
\begin{displaymath}
yR_{_{(u\backslash([(uv)/(u\backslash (xv))]v))}}=yR_{[u\backslash
(xv)]}\mathbb{R}_vR_{[x\backslash (uv)]}\mathbb{R}_{[u\backslash
(xv)]}R_v~\textrm{implies that}
\end{displaymath}
\begin{displaymath}
y\{u\backslash ([(uv)/(u\backslash (xv))]v)\}=\{(y[u\backslash
(xv)])/v\cdot [x\backslash (uv)]\}/[u\backslash (xv)]\cdot v.
\end{displaymath}
\item[OSI$_{01.1}$] Consider
\begin{displaymath}
\alpha (x,u,v)=\gamma
(x,u,v)\mathbb{R}_{_{(u\backslash[(u/v)(u\backslash
(xv))])}},~\textrm{then for all $y\in Q$,}
\end{displaymath}
\begin{displaymath}
y\alpha (x,u,v)=yR_{_{(u\backslash([(uv)/(u\backslash
(xv))]v))}}\mathbb{R}_vR_{[u\backslash
(xv)]}\mathbb{L}_uL_x\mathbb{R}_v=y\gamma
(x,u,v)\mathbb{R}_{_{(u\backslash[(u/v)(u\backslash
(xv))])}}=
\end{displaymath}
\begin{displaymath}
y\mathbb{R}_vR_{[u\backslash
(xv)]}\mathbb{L}_uL_x\mathbb{R}_{_{(u\backslash[(u/v)(u\backslash
(xv))])}}.~\textrm{Consequently,}
\end{displaymath}
\begin{displaymath}
\{x\cdot u\backslash(\{y(u\backslash([(uv)/(u\backslash
(xv))]v))\}/v\cdot [u\backslash (xv)])\}/v=\{x\cdot
u\backslash(y/v\cdot [u\backslash
(xv)])\}/(u\backslash[(u/v)(u\backslash (xv))]).
\end{displaymath}
Now replace $y/v$ by $y$ and post-multiply both sides by
$(u\backslash[(u/v)(u\backslash (xv))])$ to get
\begin{displaymath}
\{x\cdot u\backslash(\{(yv)(u\backslash([(uv)/(u\backslash
(xv))]v))\}/v\cdot [u\backslash (xv)])\}/v\cdot
(u\backslash[(u/v)(u\backslash (xv))])=x\cdot u\backslash(y\cdot
[u\backslash (xv)]).
\end{displaymath}
Again, let $z=u\backslash (xv)$ which implies that $x=(uz)/v$ and
hence we now have
\begin{displaymath}
\{(uz)/v\cdot u\backslash(\{(yv)(u\backslash([(uv)/z]v))\}/v\cdot
z)\}/v\cdot (u\backslash[(u/v)z])=(uz)/v\cdot u\backslash(yz).
\end{displaymath}
\item[OSI$_{01.2}$] Consider
\begin{displaymath}
R_{[u\backslash (xv)]}\mathbb{R}_vR_{[x\backslash
(uv)]}\mathbb{R}_{[u\backslash (xv)]}R_v\gamma
(x,u,v)\mathbb{R}_v=\gamma
(x,u,v)\mathbb{R}_{_{(u\backslash[(u/v)(u\backslash
(xv))])}},~\textrm{then for all $y\in Q$,}
\end{displaymath}
\begin{displaymath}
yR_{[u\backslash (xv)]}\mathbb{R}_vR_{[x\backslash
(uv)]}\mathbb{R}_{[u\backslash (xv)]}R_v\gamma
(x,u,v)\mathbb{R}_v=y\gamma
(x,u,v)\mathbb{R}_{_{(u\backslash[(u/v)(u\backslash
(xv))])}}~\textrm{results in}
\end{displaymath}
\begin{displaymath}
(\{\big[(y[u\backslash (xv)])/v\cdot [x\backslash
(uv)]\big]/[u\backslash (xv)]\cdot v\}\gamma (x,u,v))/v=\big(y\gamma
(x,u,v)\big)/(u\backslash[(u/v)(u\backslash (xv))])
\end{displaymath}
which is equivalent to the equation below after substituting the
value of $\gamma (x,u,v)$ and post-multiply both sides by $v$:
\begin{displaymath}
x\cdot u\backslash(\{\big[(y[u\backslash (xv)])/v\cdot [x\backslash
(uv)]\big]/[u\backslash (xv)]\cdot v\}/v\cdot [u\backslash (xv)])=
\end{displaymath}
\begin{displaymath}
[x\cdot u\backslash\big(y/v\cdot [u\backslash
(xv)]\big)]/(u\backslash[(u/v)(u\backslash (xv))])\cdot v.
\end{displaymath}
Do the replacement $z=u\backslash (xv)\Rightarrow x=(uz)/v$ to get
\begin{displaymath}
(uz)/v\cdot u\backslash\big[(yz)/v\cdot [(uz)/v\backslash
(uv)]\big]=[(uz)/v\cdot u\backslash\big(y/v\cdot
z)]/(u\backslash[(u/v)z])\cdot v.
\end{displaymath}
Now, replace $y$ by $yv$ to get
\begin{displaymath}
(uz)/v\cdot u\backslash\big[(yv\cdot z)/v\cdot [(uz)/v\backslash
(uv)]\big]=[(uz)/v\cdot
u\backslash\big(yz)]/(u\backslash[(u/v)z])\cdot v.
\end{displaymath}
\end{description}
Identity OSI$_{01.1.1}$ is deduced from identity OSI$_{01.1}$ while
identities OSI$_{01.2.1}$ and OSI$_{01.2.2}$ are deduced from
identity OSI$_{01.2}$. The other identities are gotten from
OSI$_{01.1.1}$ and OSI$_{01.2.2}$.
\begin{description}
\item[OSI$_{01.1.1}$] Put $u=v$ in identity OSI$_{01.1}$ to get
\begin{displaymath}
\{(uz)/u\cdot u\backslash(\{(yu)(u\backslash([(uu)/z]u))\}/u\cdot
z)\}/u\cdot (u\backslash z)=(uz)/u\cdot u\backslash(yz).
\end{displaymath}
Now replace $z$ by $uz$ to get
\begin{displaymath}
\{(u\cdot uz)/u\cdot
u\backslash(\{(yu)(u\backslash([(uu)/(uz)]u))\}/u\cdot uz)\}/u\cdot
z=(u\cdot uz)/u\cdot u\backslash(y\cdot uz).
\end{displaymath}
Then, substitute $z=u^\rho$ and compute to have
\begin{displaymath}
\{u\backslash(\{(yu)(u\backslash(uu\cdot u))\}/u)\}/u\cdot
u^\rho=u\backslash y.
\end{displaymath}
Replacing $y$ by $uy$, finally have
\begin{displaymath}
\{u\backslash(\{(uy\cdot u)(u\backslash(uu\cdot u))\}/u)\}/u\cdot
u^\rho=y.
\end{displaymath}
\item[OSI$_{01.2.1}$] Substitute $z=u^\rho$ in identity OSI$_{01.2}$ to get
\begin{displaymath}
v^\lambda\cdot u\backslash\big[(yv\cdot u^\rho)/v\cdot
[v^\lambda\backslash (uv)]\big]=[v^\lambda\cdot
u\backslash\big(yu^\rho)]/(u\backslash[(u/v)u^\rho])\cdot v.
\end{displaymath}
\item[OSI$_{01.2.2}$] Put $u=e$ in identity OSI$_{01.2.1}$ to get
\begin{displaymath}
v^\lambda (y\cdot v^\lambda\backslash v)=(v^\lambda
y)/v^\lambda\cdot v.
\end{displaymath}
\end{description}
By putting $y=e$ in identity OSI$_{01.1.1}$, we have $uu\cdot
u\backslash (uu\cdot u)=(u\cdot uu)u$. Also, substitute $y=v$ into
identity OSI$_{01.2.2}$ and use the fact that
$x^{\lambda^2}=x^\lambda\cdot xx$ to get $v^\lambda\cdot (v\cdot
v^\lambda\backslash v)=v^{\lambda^2}\cdot v=(v^\lambda\cdot
vv)v~\textrm{and}~v(v^\rho\cdot v\backslash v^\rho )=v^\lambda\cdot
v^\rho$.\end{proof}

\begin{mylem}\label{1:9b}
A universal Osborn loop is a 3-PAPL if and only if it is a {\Large
4}$_{11\cdot 11=(1\cdot 11)1}^{1}$ and a {\Large 4}$_{11\cdot
11=(11\cdot 1)1}^{1}$ loop.
\end{mylem}
\begin{proof}
In Lemma~\ref{1:9a}, it was shown that $uu\cdot u\backslash (uu\cdot
u)=(u\cdot uu)u$ in a universal Osborn loop. The necessary and
sufficient parts are easy to prove using this identity.\end{proof}

\begin{mylem}\label{1:9c}
In a universal Osborn loop $Q$, the following are equivalent.
\begin{enumerate}
\item $Q$ is a 3-PAPL.
\item $Q$ is a {\Large 4}$_{11\cdot 11=(1\cdot 11)1}^{1}$ loop and a {\Large 4}$_{11\cdot 11=(11\cdot 1)1}^{1}$ loop.
\item $Q$ is a LSIPL.
\item $Q$ obeys the identity $v[v^\lambda\cdot (v\cdot v^\lambda\backslash v)]=v^\lambda\backslash v\cdot
v$.
\item $Q$ is a {\Large 4}$_{12\cdot 22=(1\cdot 22)2}^{1,3}$ loop.
\item $Q$ is a {\Large 4}$_{11\cdot 11=(1\cdot 11)1}^{1}$ loop.
\end{enumerate}
\end{mylem}
\begin{proof}
This is established by using the identities $uu\cdot u\backslash
(uu\cdot u)=(u\cdot uu)u$ and $v^\lambda\cdot (v\cdot
v^\lambda\backslash v)=(v^\lambda\cdot vv)v$ of Lemma~\ref{1:9};
Lemma~\ref{1:9.1b}, Lemma~\ref{1:9.1d},
Lemma~\ref{1:9.1e}.\end{proof}

\begin{mycor}\label{1:9ci}
In a universal Osborn loop, the {\Large 4}$_{11\cdot 11=(1\cdot
11)1}^{1}$ and {\Large 4}$_{11\cdot 11=(11\cdot 1)1}^{1}$ loop
properties are equivalent.
\end{mycor}
\begin{proof}
This follows from Lemma~\ref{1:9c}.\end{proof}

\begin{mycor}\label{1:9d}
A universal Osborn loop that is a LSIPL or RSIPL or 3-PAPL or
{\Large 4}$_{12\cdot 22=(1\cdot 22)2}^{1,3}$ or {\Large 4}$_{11\cdot
11=(1\cdot 11)1}^{1}$ loop is a L${}^2$BSIPL and a L${}^1$BSIPL.
\end{mycor}
\begin{proof}
This is established by using Corollary~\ref{1:9.1i},
Lemma~\ref{1:9.1c} and Lemma~\ref{1:9c}.\end{proof}

\begin{myth}\label{1:4.1}
A loop $(Q, \cdot ,\backslash ,/)$ is a left universal Osborn loop
if and only if it obeys the identity
\begin{displaymath}
\underbrace{x\cdot [(y\cdot zv)/v\cdot (xv)]=(x\cdot
\{[y([v/(xv)]v)]/v\cdot (xv)\})/v \cdot [z\cdot
xv]}_{\textrm{OS$_0^\lambda$}}\qquad\textrm{or}
\end{displaymath}
\begin{displaymath}
\underbrace{ x\cdot [(y\cdot zv)/v\cdot (xv)]=\{x\cdot [(y\cdot
xv)/v\cdot (x\backslash v)]\}/v \cdot
[z(xv)].}_{\textrm{OS$_1^\lambda$}}
\end{displaymath}
\end{myth}
\begin{proof}
The procedure of the proof of this theorem is similar to the
procedure used to prove Theorem~\ref{1:4} by just using the
arbitrary left principal isotope $\mathfrak{Q}=(Q, \vartriangle
,\nwarrow ,\nearrow )$ such that
\begin{displaymath}
x\vartriangle y=xR_v^{-1}\cdot y=(x/v)\cdot y~\forall~v\in Q.
\end{displaymath}
\end{proof}

\begin{mylem}\label{1:8.1}
Let $Q$ be a loop with multiplication group
$\mathcal{M}\textrm{ult}(Q)$. $Q$ is a left universal Osborn loop if
and only if the triple $\big(\alpha (x,v),\beta (x,v),\gamma
(x,v)\big)\in AUT(Q)$ or $\big(R_{[xv]}\mathbb{R}_vR_{[x\backslash
v]}\mathbb{R}_{[xv]}R_v\gamma (x,v)\mathbb{R}_v,\beta (x,v),\gamma
(x,v)\big)\in AUT(Q)$ for all $x,v\in Q$ where $\alpha
(x,v)=R_{([v/(xv)]v)}\mathbb{R}_vR_{[xv]}L_x\mathbb{R}_v,~\beta
(x,v)=\mathbb{R}_vR_{[xv]}$ and $\gamma
(x,v)=\mathbb{R}_vR_{[xv]}L_x$ are elements of
$\mathcal{M}\textrm{ult}(Q)$.
\end{mylem}
\begin{proof}
This is obtained from identity OS$_0^\lambda$ or OS$_1^\lambda$ of
Theorem~\ref{1:4.1}.\end{proof}

\begin{myth}\label{1:9.1}
Let $Q$ be a loop with multiplication group
$\mathcal{M}\textrm{ult}(Q)$. If $Q$ is a left universal Osborn
loop, then the triple $\Big(\gamma (x,v)\mathbb{R}_{[v^\lambda\cdot
xv]},\beta (x,v),\gamma (x,v)\Big)\in AUT(Q)$ for all $x,v\in Q$
where $\beta (x,v)=\mathbb{R}_vR_{(xv)}$ and $\gamma
(x,v)=\mathbb{R}_vR_{(xv)}L_x$ are elements of
$\mathcal{M}\textrm{ult}(Q)$.
\end{myth}
\begin{proof}
This follows by using identity OS$_0^\lambda$ or OS$_1^\lambda$ of
Theorem~\ref{1:4.1} the way identity OS$_0'$ or OS$_1'$ of
Theorem~\ref{1:4} was used to prove Theorem~\ref{1:9}.\end{proof}

\begin{mylem}\label{1:9.1a}
Let $(Q, \cdot ,\backslash ,/)$  be a left universal Osborn loop.
The following identities are satisfied:
\begin{gather*}
\underbrace{ y\{[v/(xv)]v\}=\{[y(xv)]/v\cdot (x\backslash
v)\}/(xv)\cdot
v}_{\textrm{OSI$_{01}^\lambda$}},~\underbrace{z\{(yv\cdot zv)/v\cdot
z\backslash v\}=[z(y\cdot
zv)]/(v^\lambda\cdot zv)\cdot v}_{\textrm{OSI$_{01.2}^\lambda$}}~\textrm{and}\\
\underbrace{\{z\cdot \{[(yv)(v/(zv)\cdot v)]/v\cdot zv\}\}/v\cdot
v^\lambda (zv)=z\cdot
y(vz)}_{\textrm{OSI$_{01.1}^\lambda$}}.\\
\textrm{Furthermore},~\underbrace{\{v^\lambda\{[(yv)(vv)]/v\}\}/v\cdot
v^\lambda=v^\lambda y}_{\textrm{OSI$_{01.1.1}^\lambda$}},~
\underbrace{\{z\{[v(v/(zv)\cdot v)]z\}\}/v\cdot v^\lambda
(zv)=z\cdot
zv}_{\textrm{OSI$_{01.1.2}^\lambda$}},\\\underbrace{v\{(yv\cdot
vv)/v\}=[v(y\cdot vv)]/(v^\lambda\cdot vv)\cdot
v}_{\textrm{OSI$_{01.2.1}^\lambda$}},~\underbrace{v[(v\cdot
vv)/v]=(v\cdot vv)/(v^\lambda\cdot vv)\cdot
v}_{\textrm{OSI$_{01.2.2}^\lambda$}},\\
\underbrace{v[(vv\cdot vv)/v]=[v(v\cdot vv)]/(v^\lambda\cdot
vv)\cdot
v}_{\textrm{OSI$_{01.2.3}^\lambda$}},~\underbrace{v^\lambda[y\cdot
v^\lambda\backslash v]=(v^\lambda y)/v^\lambda\cdot
v}_{\textrm{OSI$_{01.2.4}^\lambda$}},\\
v\cdot vv=v^\lambda\backslash v\cdot v~\textrm{and}~vv\cdot
vv=v^\lambda\backslash (v^{\lambda^2}v)\cdot v
\end{gather*}
are also satisfied.
\end{mylem}
\begin{proof}
To prove these identities, we shall make use of the three
autotopisms in Lemma~\ref{1:8.1} and Theorem~\ref{1:9.1}. In a
quasigroup, any two components of an autotopism uniquely determine
the third. So equating the first components of the three
autotopisms, it is easy to see that
\begin{displaymath}
\alpha (x,v)=\gamma (x,v)\mathbb{R}_{[v^\lambda \cdot
xv]}=R_{[xv]}\mathbb{R}_vR_{[x\backslash
v]}\mathbb{R}_{[xv]}R_v\gamma (x,v)\mathbb{R}_v.
\end{displaymath}
The establishment of the identities OSI$_{01}^\lambda$,
OSI$_{01.1}^\lambda$ and OSI$_{01.2}^\lambda$ follows by using the
bijections appropriately to map an arbitrary element $y\in Q$ as
follows:
\begin{description}
\item[OSI$_{01}^\lambda$]
\begin{displaymath}
\alpha (x,v)=R_{[xv]}\mathbb{R}_vR_{[x\backslash
v]}\mathbb{R}_{[xv]}R_v\gamma (x,v)\mathbb{R}_v~\textrm{implies
that}
\end{displaymath}
\begin{displaymath}
R_{([v/(xv)]v)}\mathbb{R}_vR_{[xv]}L_x\mathbb{R}_v=R_{([v/(xv)]v)}\gamma
(x,v)\mathbb{R}_v=R_{[xv]}\mathbb{R}_vR_{[x\backslash
v]}\mathbb{R}_{[xv]}R_v\gamma (x,v)\mathbb{R}_v
\end{displaymath}
\begin{displaymath}
\textrm{which
gives}~R_{([v/(xv)]v)}=R_{[xv]}\mathbb{R}_vR_{[x\backslash
v]}\mathbb{R}_{[xv]}R_v.
\end{displaymath}
\begin{displaymath}
\textrm{So, for any $y\in
Q$,}~yR_{([v/(xv)]v)}=yR_{[xv]}\mathbb{R}_vR_{[x\backslash
v]}\mathbb{R}_{[xv]}R_v~\textrm{implies that}
\end{displaymath}
\begin{displaymath}
y\{[v/(xv)]v\}=\{[y(xv)]/v\cdot (x\backslash v)\}/(xv)\cdot v
\end{displaymath}
\item[OSI$_{01.1}^\lambda$] Consider
\begin{displaymath}
\alpha (x,v)=\gamma (x,v)\mathbb{R}_{[v^\lambda \cdot
xv]},~\textrm{then for all $y\in Q$,}
\end{displaymath}
\begin{displaymath}
y\alpha
(x,v)=yR_{([v/(xv)]v)}\mathbb{R}_vR_{[xv]}L_x\mathbb{R}_v=y\gamma
(x,v)\mathbb{R}_{[v^\lambda \cdot
xv]}=y\mathbb{R}_vR_{(xv)}L_x\mathbb{R}_{[v^\lambda \cdot xv]}.
\end{displaymath}
\begin{displaymath}
\textrm{Consequently,}~\{x\cdot (\{y([v/(xv)]v)\}/v\cdot
xv)\}/v=\{x\cdot (y/v\cdot xv)\}/[v^\lambda\cdot xv].
\end{displaymath}
Now replace $y/v$ by $y$ and post-multiply both sides by
$[v^\lambda\cdot xv]$ to get
\begin{displaymath}
\{x\cdot (\{(yv)([v/(xv)]v)\}/v\cdot xv)\}/v\cdot [v^\lambda\cdot
xv]=\{x\cdot (y\cdot xv)\}.
\end{displaymath}
\item[OSI$_{01.2}^\lambda$] Consider
\begin{displaymath}
R_{[xv]}\mathbb{R}_vR_{[x\backslash v]}\mathbb{R}_{[xv]}R_v\gamma
(x,v)\mathbb{R}_v=\gamma (x,v)\mathbb{R}_{[v^\lambda \cdot
xv]},~\textrm{then for all $y\in Q$,}
\end{displaymath}
\begin{displaymath}
yR_{[xv]}\mathbb{R}_vR_{[x\backslash v]}\mathbb{R}_{[xv]}R_v\gamma
(x,v)\mathbb{R}_v=y\gamma (x,v)\mathbb{R}_{[v^\lambda \cdot
xv]}~\textrm{results in}
\end{displaymath}
\begin{displaymath}
(\{\big[[y(xv)]/v\cdot (x\backslash v)\big]/(xv)\cdot v\}\gamma
(x,v))/v=\big(y\gamma (x,v)\big)/[v^\lambda \cdot xv]
\end{displaymath}
which is equivalent to the equation below after substituting the
value of $\gamma (x,v)$ and post-multiply both sides by $v$:
\begin{displaymath}
x\{[y(xv)]/v\cdot (x\backslash v)\}=(x\cdot [y/v\cdot
(xv)])/[v^\lambda \cdot xv]\cdot v.
\end{displaymath}
\begin{displaymath}
\textrm{Now, replace $y$ by $yv$ to get}~x\{[(yz)(xv)]/v\cdot
(x\backslash v)\}=(x[y\cdot (xv)])/[v^\lambda \cdot xv]\cdot v.
\end{displaymath}
\end{description}
Identities OSI$_{01.1.1}^\lambda$ and OSI$_{01.1.2}^\lambda$ are
deduced from identity OSI$_{01.1}^\lambda$. Identities
OSI$_{01.2.1}^\lambda$ and OSI$_{01.2.4}^\lambda$ are deduced from
identity OSI$_{01.2}^\lambda$ while identities
OSI$_{01.2.2}^\lambda$ and OSI$_{01.2.3}^\lambda$ are deduced from
identity OSI$_{01.2.1}^\lambda$. The other identities are gotten
from OSI$_{01.1.1}^\lambda$.
\begin{description}
\item[OSI$_{01.1.1}^\lambda$] Simply put $z=v^\lambda$ in identity OSI$_{01.1}^\lambda$ to
get identity OSI$_{01.1.1}^\lambda$.
\item[OSI$_{01.1.2}^\lambda$] Simply put $y=e$ in identity OSI$_{01.1}^\lambda$ to
get identity OSI$_{01.1.2}^\lambda$.
\item[OSI$_{01.2.1}^\lambda$] Simply put $z=v$ in identity OSI$_{01.2}^\lambda$ to
get identity OSI$_{01.2.1}^\lambda$.
\item[OSI$_{01.2.2}^\lambda$] Substitute $y=e$ in identity OSI$_{01.2.1}^\lambda$ to
get identity OSI$_{01.2.2}^\lambda$.
\item[OSI$_{01.2.3}^\lambda$] Substitute $y=v$ in identity OSI$_{01.2.1}^\lambda$ to
get identity OSI$_{01.2.3}^\lambda$.
\item[OSI$_{01.2.4}^\lambda$] Simply put $z=v^\lambda$ in identity OSI$_{01.2}^\lambda$ to
get identity OSI$_{01.2.4}^\lambda$.
\end{description}
By putting $y=e$ in identity OSI$_{01.1.1}^\lambda$, we have
$\{v^\lambda\{[v(vv)]/v\}\}/v\cdot v^\lambda=v^\lambda $ which
implies $v^\lambda\{[v(vv)]/v\}=v$, hence,
$v(vv)=(v^\lambda\backslash v)\cdot v$.

Again, by putting $y=v$ in identity OSI$_{01.1.1}^\lambda$, we have
$\{v^\lambda\{[(vv)(vv)]/v\}\}/v\cdot v^\lambda=e$ which implies
$v^\lambda\{[(vv)(vv)]/v\}=v^{\lambda^2}v$, hence, $vv\cdot
vv=v^\lambda\backslash (v^{\lambda^2}v)\cdot v$.\end{proof}

\begin{mylem}\label{1:9.1b}
A left universal Osborn loop is a LSIPL if and only if it is a 3
PAPL.
\end{mylem}
\begin{proof}
This is proved by using the identity $v\cdot vv=v^\lambda\backslash
v\cdot v$ in Lemma~\ref{1:9.1a}.\end{proof}

\begin{mylem}\label{1:9.1c}
A left universal Osborn loop $(Q, \cdot ,\backslash ,/)$ is a
{\Large 4}$_{11\cdot 11=(11\cdot 1)1}^{1}$ loop if and only if it
obeys the identity $v^\lambda (vv\cdot v)=v^{\lambda^2}v$.
\end{mylem}
\begin{proof}
This is proved by using the identity $vv\cdot vv=v^\lambda\backslash
(v^{\lambda^2}v)\cdot v$ in Lemma~\ref{1:9.1a}.\end{proof}

\begin{mycor}\label{1:9.1ci}
A left universal Osborn loop $(Q, \cdot ,\backslash ,/)$ that is a
{\Large 4}$_{11\cdot 11=(11\cdot 1)1}^{1}$ loop is a L${}^1$BSIPL if
and only if it is a LSIPL. Hence, it is a L${}^2$BSIPL.
\end{mycor}
\begin{proof}
This follows from Lemma~\ref{1:9.1c} by using the fact that in an
Osborn loop, $x^{\lambda^2}=x\mapsto x^\lambda\cdot xx$.\end{proof}

\begin{mylem}\label{1:9.1d}
A left universal Osborn loop is a LSIPL if and only if it is a
{\Large 4}$_{12\cdot 22=(1\cdot 22)2}^{1,3}$ loop.
\end{mylem}
\begin{proof}
This is proved by using the identity OSI$_{01.2.1}^\lambda$ of
Lemma~\ref{1:9.1a}.\end{proof}

\begin{mylem}\label{1:9.1e}
A left universal Osborn loop is a LSIPL if and only if it is a
{\Large 4}$_{11\cdot 11=(1\cdot 11)1}^{1}$ loop.
\end{mylem}
\begin{proof}
This is proved by using the identity OSI$_{01.2.3}^\lambda$ of
Lemma~\ref{1:9.1a}.\end{proof}

\begin{mylem}\label{1:9.1f}
Let $G$ be a left universal Osborn loop. The following are
equivalent.
\begin{enumerate}
\item $G$ is a LSIPL and a {\Large 4}$_{12\cdot 22=(12\cdot 2)2}^{1,3}$ loop.
\item $G$ is a left alternative property loop.
\item $G$ is a Moufang loop.
\end{enumerate}
\end{mylem}
\begin{proof}
This is proved by using the identity OSI$_{01.2.1}^\lambda$ of
Lemma~\ref{1:9.1a}.\end{proof}

\begin{mylem}\label{1:9.1g}
A left universal Osborn loop $(Q, \cdot ,\backslash ,/)$ that is a
{\Large 4}$_{12\cdot 12=(12\cdot 1)2}^{2,2}$ or {\Large 4}$_{12\cdot
12=(1\cdot 21)2}^{2,2}$ loop obeys the identity $[y(yy\cdot y^\rho
)]y=y\cdot yy$.
\end{mylem}
\begin{proof}
This is proved by using the identity OSI$_{01.2}^\lambda$ of
Lemma~\ref{1:9.1a}.\end{proof}

\begin{mylem}\label{1:9.1h}
In an Osborn loop, the following properties are equivalent. LSIP,
RSIP, $|J_\lambda |=2$, $|J_\rho|=2$ and $J_\rho =J_\lambda$.
\end{mylem}
\begin{proof}
This can be proved by using the facts that in an Osborn loop,
$J_\rho^2~:~x\mapsto xx\cdot x^\rho$ and $J_\lambda^2~:~x\mapsto
x^\lambda\cdot xx$.\end{proof}

\begin{mycor}\label{1:9.1i}
In a left universal Osborn loop, the following properties are
equivalent. LSIP, RSIP, 3-PAP, $J_\rho =J_\lambda$, {\Large
4}$_{12\cdot 22=(1\cdot 22)2}^{1,3}$ and {\Large 4}$_{11\cdot
11=(1\cdot 11)1}^{1}$ properties.
\end{mycor}
\begin{proof}
Use Lemma~\ref{1:9.1h}, Lemma~\ref{1:9.1b}, Lemma~\ref{1:9.1d} and
Lemma~\ref{1:9.1e}.\end{proof}

\begin{mycor}\label{1:9.1j}
Let $L$ be a CC-loop. The following are equivalent.
\begin{enumerate}
\item $L$ is a power associativity loop.
\item $L$ is a 3-PAPL.
\item $L$ obeys $x^\rho =x^\lambda$ for all $x\in L$.
\item $L$ is a LSIPL.
\item $L$ is a RSIPL.
\item $L$ is a {\Large 4}$_{12\cdot 22=(1\cdot 22)2}^{1,3}$ loop .
\item $L$ is a {\Large 4}$_{11\cdot 11=(1\cdot 11)1}^{1}$ loop.
\end{enumerate}
\end{mycor}
\begin{proof}
The proof of the equivalence of the first three is shown in
Lemma~3.20 of \cite{phd78} and mentioned in Lemma~1.2 of
\cite{phd151}. The proof of the equivalence of the last four and the
first three can be deduced from the last result of
Corollary~\ref{1:9.1i}.\end{proof}

\begin{mycor}\label{1:9.1k}
A CC-loop is a diassociative loop if and only if it is a power
associative loop and a {\Large 4}$_{12\cdot 22=(12\cdot 2)2}^{1,3}$
loop.
\end{mycor}
\begin{proof}
The proof of this follows from Corollary~\ref{1:9.1j} and
Lemma~\ref{1:9.1f}.\end{proof}

\begin{myth}\label{1:4.11}
A loop $(Q, \cdot ,\backslash ,/)$ is a right universal Osborn loop
if and only if it obeys the identity
\begin{displaymath}
\underbrace{(ux)\cdot u\backslash\{yz\cdot x\}=((ux)\cdot
u\backslash\{[y(u\backslash[u/x])]\cdot x\}) \cdot u\backslash[(uz)
x]}_{\textrm{OS$_0^\rho$}}\qquad\textrm{or}
\end{displaymath}
\begin{displaymath}
\underbrace{(ux)\cdot u\backslash\{(yz)\cdot x\}=\{(ux)\cdot
u\backslash[yx\cdot (ux)\backslash u]\}\cdot
u\backslash[(uz)x].}_{\textrm{OS$_1^\rho$}}
\end{displaymath}
\end{myth}
\begin{proof}
The procedure of the proof of this theorem is similar to the
procedure used to prove Theorem~\ref{1:4} by just using the
arbitrary right principal isotope $\mathfrak{Q}=(Q, \vartriangle
,\nwarrow ,\nearrow )$ such that
\begin{displaymath}
x\vartriangle y=x\cdot yL_u^{-1}=x\cdot (u\backslash y)~\forall~u\in
Q.
\end{displaymath}
\end{proof}

\begin{mylem}\label{1:8.11}
Let $Q$ be a loop with multiplication group
$\mathcal{M}\textrm{ult}(Q)$. $Q$ is a right universal Osborn loop
if and only if the triple $\big(\alpha (x,u),\beta (x,u),\gamma
(x,u)\big)\in AUT(Q)$ or the triple $\Big(R_{[u\backslash
x]}R_{[x\backslash u]}\mathbb{R}_{[u\backslash x]}\gamma (x,u),\beta
(x,u),\gamma (x,u)\Big)\in AUT(Q)$ for all $x,u\in Q$ where $\alpha
(x,u)=R_{(u\backslash [u/(u\backslash x)])}R_{[u\backslash
x]}\mathbb{L}_uL_x,~\beta (x,u)=L_uR_{[u\backslash x]}\mathbb{L}_u$
and $\gamma (x,u)=R_{[u\backslash x]}\mathbb{L}_uL_x$ are elements
of $\mathcal{M}\textrm{ult}(Q)$.
\end{mylem}
\begin{proof}
This is obtained by using identity OS$_0^\rho$ or OS$_1^\rho$ of
Theorem~\ref{1:4.11}.\end{proof}

\begin{myth}\label{1:9.11}
Let $Q$ be a loop with multiplication group
$\mathcal{M}\textrm{ult}(Q)$. If $Q$ is a right universal Osborn
loop, then the triple $\Big(\gamma (x,u)\mathbb{R}_{(u\backslash
x)},\beta (x,u),\gamma (x,u)\Big)\in AUT(Q)$ for all $x,u\in Q$
where $\beta (x,u)=L_uR_{[u\backslash x)]}\mathbb{L}_u$ and $\gamma
(x,u)=R_{[u\backslash x]}\mathbb{L}_uL_x$ are elements of
$\mathcal{M}\textrm{ult}(Q)$.
\end{myth}
\begin{proof}
This follows by using identity OS$_0^\rho$ or OS$_1^\rho$ in
Theorem~\ref{1:4.11} the way identity OS$_0'$ or OS$_1'$ was used in
Theorem~\ref{1:4}.\end{proof}

\begin{mylem}\label{1:9.11a}
Let $(Q, \cdot ,\backslash ,/)$ be a right universal Osborn loop.
The following identities are satisfied:
\begin{gather*}
\underbrace{ y\{u\backslash (u/x)\}=\{(yx)\cdot [(ux)\backslash
u]\}/x}_{\textrm{OSI$_{01}^\rho$}},~\underbrace{\{(uz)\cdot
u\backslash [(yz)[(uz)\backslash u]]\}z=(uz)\cdot
u\backslash (yz)}_{\textrm{OSI$_{01.2}^\rho$}}~\textrm{and}\\
\underbrace{\{(uz)\cdot u\backslash (\{y(u\backslash(u/z))\}\cdot
z)\}z=(uz)\cdot u\backslash (yz)}_{\textrm{OSI$_{01.1}^\rho$}}.\\
\textrm{Furthermore},~\underbrace{\{(uz)\cdot u\backslash
(\{z^\lambda (u\backslash(u/z))\}\cdot z)\}z=(uz)\cdot
u^\rho)}_{\textrm{OSI$_{01.1.1}^\rho$}},~\underbrace{\{(uu)\cdot
u\backslash (u^\lambda u^\rho\cdot u)\}u=uu\cdot
u^\rho)}_{\textrm{OSI$_{01.1.2}^\rho$}},\\
\underbrace{\{(uz)\cdot u\backslash (\{z(u\backslash(u/z))\}\cdot
z)\}z=(uz)\cdot u\backslash (zz)}_{\textrm{OSI$_{01.1.3}^\rho$}},~
\underbrace{\{u\backslash (\{u^\rho (u\backslash(u/u^\rho ))\}\cdot
u^\rho )\}z=u\backslash (u^\rho u^\rho )}_{\textrm{OSI$_{01.1.4}^\rho$}},\\
\underbrace{\{(uz)\cdot u\backslash (\{z^\rho
(u\backslash(u/z))\}\cdot z)\}z=(uz)\cdot u\backslash
(z^\rho z)}_{\textrm{OSI$_{01.1.5}^\rho$}},\\
\underbrace{z^\lambda\backslash [\{z^\rho
(z^\lambda\backslash(z^\lambda/z))\}\cdot z]\cdot
z=z^\lambda\backslash (z^\rho
z)}_{\textrm{OSI$_{01.1.6}^\rho$}},\qquad \underbrace{\{(zz)\cdot
z\backslash (z^\rho z^\rho\cdot z)\}z=(zz)\cdot z\backslash (z^\rho
z)}_{\textrm{OSI$_{01.1.7}^\rho$}},
\\\underbrace{\{(uz)\cdot u\backslash [(uz)\backslash u]\}z=(uz)\cdot
u^\rho}_{\textrm{OSI$_{01.2.1}^\rho$}},~\underbrace{\{(uu)\cdot
u\backslash [(uu)\backslash u]\}u=(uu)\cdot
u^\rho}_{\textrm{OSI$_{01.2.2}^\rho$}},\\
\underbrace{\{(uu^\lambda )\cdot u\backslash [(uu^\lambda
)\backslash u]\}u^\lambda =(uu^\lambda )\cdot
u^\rho}_{\textrm{OSI$_{01.2.3}^\rho$}},~\underbrace{\{(uz)\cdot
u\backslash [z[(uz)\backslash u]]\}z=(uz)\cdot
u\backslash z}_{\textrm{OSI$_{01.2.4}^\rho$}},\\
\underbrace{\{(uu^\lambda )\cdot u\backslash [u^\lambda [(uu^\lambda
)\backslash u]]\}u^\lambda =(uu^\lambda )\cdot
u\backslash uu^\lambda }_{\textrm{OSI$_{01.2.5}^\rho$}},\\
\underbrace{\{(uz)\cdot u\backslash [(zz)[(uz)\backslash
u]]\}z=(uz)\cdot u\backslash
(zz)}_{\textrm{OSI$_{01.2.6}^\rho$}},~\underbrace{\{(uz)\cdot
u\backslash [(z^\rho z)[(uz)\backslash u]]\}z=(uz)\cdot u\backslash
(z^\rho z)}_{\textrm{OSI$_{01.2.7}^\rho$}},\\
\underbrace{\{(uu)\cdot u\backslash [(u^\rho u)[(uu)\backslash
u]]\}u=(uu)\cdot u\backslash (u^\rho
u)}_{\textrm{OSI$_{01.2.8}^\rho$}},~\underbrace{(uu\cdot u^\rho
)u^\rho=u\{u\backslash[(uu\cdot u^\rho )u^\rho\cdot u]\cdot u^\rho
\}}_{\textrm{OSI$_{01.2.9}^\rho$}},\\
\underbrace{\{(uu^\lambda )\cdot u\backslash [(uu^\lambda
)[(uu^\lambda )\backslash u]]\}u^\lambda =(uu^\lambda )\cdot
u\backslash (uu^\lambda
)}_{\textrm{OSI$_{01.2.10}^\rho$}}~\textrm{and}~u\cdot[u\backslash
(u^\rho u)]u^\rho =u^\rho
\end{gather*}
are also satisfied.
\end{mylem}
\begin{proof}
To prove these identities, we shall make use of the three
autotopisms in Lemma~\ref{1:8.11} and Theorem~\ref{1:9.11}. In a
quasigroup, any two components of an autotopism uniquely determine
the third. So equating the first components of the three
autotopisms, it is easy to see that
\begin{displaymath} \alpha
(x,u)=\gamma (x,u)\mathbb{R}_{(u\backslash x)}=R_{[u\backslash
x]}R_{[x\backslash u]}\mathbb{R}_{[u\backslash x]}\gamma (x,u).
\end{displaymath}
The establishment of the identities OSI$_{01}^\rho$,
OSI$_{01.1}^\rho$ and OSI$_{01.2}^\rho$ follows by using the
bijections to map an arbitrary element $y\in Q$ as follows:
\begin{description}
\item[OSI$_{01}^\rho$]
\begin{displaymath}
\alpha (x,u)=R_{[u\backslash x]}R_{[x\backslash
u]}\mathbb{R}_{[u\backslash x]}\gamma (x,u)~\textrm{implies that}
\end{displaymath}
\begin{displaymath}
R_{(u\backslash [u/(u\backslash x)])}R_{[u\backslash
x]}\mathbb{L}_uL_x=R_{[u\backslash x]}R_{[x\backslash
u]}\mathbb{R}_{[u\backslash x]}\gamma (x,u)=R_{[u\backslash
x]}R_{[x\backslash u]}\mathbb{R}_{[u\backslash x]}R_{[u\backslash
x]}\mathbb{L}_uL_x
\end{displaymath}
\begin{displaymath}
\textrm{which gives}~R_{(u\backslash [u/(u\backslash
x)])}=R_{[u\backslash x]}R_{[x\backslash u]}\mathbb{R}_{[u\backslash
x]}.~\textrm{So, for any $y\in Q$,}
\end{displaymath}
\begin{displaymath}
yR_{(u\backslash [u/(u\backslash x)])}=yR_{[u\backslash
x]}R_{[x\backslash u]}\mathbb{R}_{[u\backslash x]}~\textrm{implies
that}~y(u\backslash [u/z])=\{(yz)[(uz)\backslash u]\}/z.
\end{displaymath}
\begin{displaymath}
\textrm{Let $z=u\backslash x$ so that $x=uz$. Thus,}~y(u\backslash
[u/(u\backslash x)])=\{[y(u\backslash x)][x\backslash
u]\}/[u\backslash x].
\end{displaymath}
\item[OSI$_{01.1}^\rho$] Consider
\begin{displaymath}
\alpha (x,u)=\gamma (x,u)\mathbb{R}_{(u\backslash x)},~\textrm{then
for all $y\in Q$,}
\end{displaymath}
\begin{displaymath}
y\alpha (x,u)=yR_{(u\backslash [u/(u\backslash x)])}R_{[u\backslash
x]}\mathbb{L}_uL_x=y\gamma (x,u)\mathbb{R}_{(u\backslash
x)}=yR_{[u\backslash x]}\mathbb{L}_uL_x\mathbb{R}_{(u\backslash x)}.
\end{displaymath}
\begin{displaymath}
\textrm{Consequently,}~x\cdot u\backslash\{[y(u\backslash
[u/(u\backslash x)])][u\backslash x]\}=\{x\cdot u\backslash
[y(u\backslash x)]\}/(u\backslash x).
\end{displaymath}
Post-multiply both sides by $(u\backslash x)$ to get
\begin{displaymath}
\{x\cdot u\backslash\{[y(u\backslash [u/(u\backslash
x)])][u\backslash x]\}\}(u\backslash x)=x\cdot u\backslash
[y(u\backslash x)].
\end{displaymath}
Again, let $z=u\backslash x$ which implies that $x=uz$ and hence we
now have
\begin{displaymath}
\{(uz)\cdot u\backslash\{[y(u\backslash [u/z])]z\}\}z=(uz)\cdot
u\backslash (yz).
\end{displaymath}
\item[OSI$_{01.2}^\rho$] Consider
\begin{displaymath}
R_{[u\backslash x]}R_{[x\backslash u]}\mathbb{R}_{[u\backslash
x]}\gamma (x,u)=\gamma (x,u)\mathbb{R}_{(u\backslash
x)},~\textrm{then for all $y\in Q$,}
\end{displaymath}
\begin{displaymath}
yR_{[u\backslash x]}R_{[x\backslash u]}\mathbb{R}_{[u\backslash
x]}\gamma (x,u)=y\gamma (x,u)\mathbb{R}_{(u\backslash
x)}~\textrm{results in}
\end{displaymath}
\begin{displaymath}
\{\{[y(u\backslash x)](x\backslash u)\}/[u\backslash x]\}\gamma
(x,u)=[y\gamma (x,u)]/(u\backslash x)
\end{displaymath}
which is equivalent to the equation below after substituting the
value of $\gamma (x,u)$ and post multiplying by $(u\backslash x)$:
\begin{displaymath}
x\cdot u\backslash (\{\{[y(u\backslash x)](x\backslash
u)\}/[u\backslash x]\}[u\backslash x])=\{x\cdot u\backslash
(y[u\backslash x])\}/(u\backslash x).
\end{displaymath}
\begin{displaymath}
\textrm{Do the replacement $z=u\backslash x\Rightarrow x=uz$ to
get}~\{(uz)\cdot u\backslash ([(yz)((uz)\backslash u)])\}z
=(uz)\cdot u\backslash (yz).
\end{displaymath}
\end{description}
Identities OSI$_{01.1.1}^\rho$, OSI$_{01.1.3}^\rho$ and
OSI$_{01.1.5}^\rho$ are deduced from identity OSI$_{01.1}^\rho$,
identity OSI$_{01.1.2}^\rho$ is deduced from OSI$_{01.1.1}^\rho$,
identity OSI$_{01.1.4}^\rho$ is deduced from OSI$_{01.1.3}^\rho$
while identities OSI$_{01.1.6}^\rho$ and OSI$_{01.1.7}^\rho$ are
deduced from OSI$_{01.1.5}^\rho$ by doing the following:
\begin{description}
\item[OSI$_{01.1.1}^\rho$] Put $y=z^\lambda$ in identity OSI$_{01.1}^\rho$.
\item[OSI$_{01.1.2}^\rho$] Substitute $z=u$ in identity OSI$_{01.1.1}^\rho$.
\item[OSI$_{01.1.3}^\rho$] Put $y=z$ in identity OSI$_{01.1}^\rho$.
\item[OSI$_{01.1.4}^\rho$] Put $z=u^\rho$ in identity
OSI$_{01.1.3}^\rho$.
\item[OSI$_{01.1.5}^\rho$] Put $y=z^\rho$ in identity
OSI$_{01.1}^\rho$.
\item[OSI$_{01.1.6}^\rho$] Put $u=z^\lambda$ in identity
OSI$_{01.1.5}^\rho$.
\item[OSI$_{01.1.7}^\rho$] Put $u=z$ in identity
OSI$_{01.1.5}^\rho$.
\end{description}
Identities OSI$_{01.2.1}^\rho$, OSI$_{01.2.4}^\rho$,
OSI$_{01.2.6}^\rho$ and OSI$_{01.2.7}^\rho$ are deduced from
identity OSI$_{01.2}^\rho$. Identities OSI$_{01.2.2}^\rho$ and
OSI$_{01.2.3}^\rho$ are deduced from identity OSI$_{01.2.1}^\rho$.
Identity OSI$_{01.2.5}^\rho$ is deduced from identity
OSI$_{01.2.4}^\rho$. Identities OSI$_{01.2.8}^\rho$,
OSI$_{01.2.9}^\rho$ and OSI$_{01.2.10}^\rho$ are deduced from
identity OSI$_{01.2.7}^\rho$. The following are the deductions.
\begin{description}
\item[OSI$_{01.2.1}^\rho$] Put $y=z^\lambda$ in identity
OSI$_{01.2}^\rho$.
\item[OSI$_{01.2.2}^\rho$] Substitute $z=u$ in identity OSI$_{01.2.1}^\rho$.
\item[OSI$_{01.2.3}^\rho$] Put $z=u^\lambda$ in identity
OSI$_{01.2.1}^\rho$.
\item[OSI$_{01.2.4}^\rho$] Substitute $y=e$ in identity OSI$_{01.2}^\rho$.
\item[OSI$_{01.2.5}^\rho$] Put $z=u^\lambda$ in identity
OSI$_{01.2.4}^\rho$.
\item[OSI$_{01.2.6}^\rho$] Put $y=z$ in identity
OSI$_{01.2}^\rho$.
\item[OSI$_{01.2.7}^\rho$] Substitute $y=z^\rho$ in identity OSI$_{01.2}^\rho$.
\item[OSI$_{01.2.8}^\rho$] Put $z=u$ in identity
OSI$_{01.2.7}^\rho$.
\item[OSI$_{01.2.9}^\rho$] Put $z=u^\rho$ in identity
OSI$_{01.2.7}^\rho$.
\item[OSI$_{01.2.10}^\rho$] Substitute $z=u^\lambda$ in identity OSI$_{01.2.7}^\rho$.
\item Put $z=u^\rho$ in identity
OSI$_{01.2.4}^\rho$ to get $u\cdot[u\backslash (u^\rho u)]u^\rho
=u^\rho$.
\end{description}
\end{proof}

\begin{mylem}\label{1:9.11b}
A right universal Osborn loop $(Q, \cdot ,\backslash ,/)$ is a RSIPL
if and only if it obeys the identity $u^\lambda u^\rho\cdot
u=u(uu)^\rho$.
\end{mylem}
\begin{proof}
This is proved by using the identity OSI$_{01.1.2}^\rho$ in
Lemma~\ref{1:9.11a}.\end{proof}

\begin{mylem}\label{1:9.11c}
A right universal Osborn loop $(Q, \cdot ,\backslash ,/)$ is a RSIPL
if and only if it obeys the identity $u^\rho u^\rho =u[u\backslash
(u^\rho u\cdot u^\rho )\cdot u^\rho ]$.
\end{mylem}
\begin{proof}
This is proved by using the identity OSI$_{01.1.4}^\rho$ in
Lemma~\ref{1:9.11a}.\end{proof}

\begin{mylem}\label{1:9.11d}
A right universal Osborn loop $(Q, \cdot ,\backslash ,/)$ is a RSIPL
if and only if it is a $z^\lambda\backslash [z^\rho z^\lambda\cdot
z]\cdot z=z^\lambda\backslash (z^\rho z)$.
\end{mylem}
\begin{proof}
This is proved by using the identity OSI$_{01.1.6}^\rho$ of
Lemma~\ref{1:9.11a}.\end{proof}

\begin{mylem}\label{1:9.11e}
A right universal Osborn loop $(Q, \cdot ,\backslash ,/)$ obeys the
identity $zz\cdot z^\lambda =z$ if and only if it obeys the identity
$[zz\cdot z\backslash z^\rho ]z=zz\cdot z\backslash (z^\rho z)$.
\end{mylem}
\begin{proof}
This is proved by using the identity OSI$_{01.1.8}^\rho$ of
Lemma~\ref{1:9.11a}.\end{proof}

\begin{mycor}\label{1:9.11ea}
If a right universal Osborn loop $(Q, \cdot ,\backslash ,/)$ obeys
the identity $[zz\cdot z\backslash z^\rho ]z=zz\cdot z\backslash
(z^\rho z)$ then, it is a SFAIPL if and only if it is a SWIPL.
\end{mycor}
\begin{proof}
This is proved by using the identity OSI$_{01.1.8}^\rho$ of
Lemma~\ref{1:9.11a}.\end{proof}

\begin{mylem}\label{1:9.11f}
A right universal Osborn loop $(Q, \cdot ,\backslash ,/)$ with the
RSIP is a SFAIPL if and only if it obeys the identity $u\cdot
u[u\backslash u^\rho\cdot u^\rho ]=u^\rho$.
\end{mylem}
\begin{proof}
This is proved by using Lemma~\ref{1:9.11b}, Lemma~\ref{1:9.11c} and
Lemma~\ref{1:9.1h}.\end{proof}

\begin{mylem}\label{1:9.11g}
A right universal Osborn loop $(Q, \cdot ,\backslash ,/)$ with the
RSIP obeys the identity $u\backslash u^\rho =(uu)^\rho$.
\end{mylem}
\begin{proof}
This is proved by using the identity OSI$_{01.2.2}^\rho$ of
Lemma~\ref{1:9.11a}.\end{proof}

\begin{mycor}\label{1:9.11h}
A right universal Osborn loop $(Q, \cdot ,\backslash ,/)$ with the
RSIP is a SFAIPL and $|J_\rho|=6$.
\end{mycor}
\begin{proof}
This is achieved by using Lemma~\ref{1:9.11g} and
Lemma~\ref{1:9.1h}. The second claim can be deduced from the fact in
[Page 18, \cite{phd33}] that SFAIPL implies
$x^{\rho\rho\rho\rho\rho\rho}=x$.\end{proof}

\begin{mylem}\label{1:9.11i}
A right universal Osborn loop $(Q, \cdot ,\backslash ,/)$ obeys the
identity $uu^\lambda\cdot u^\rho =u^\lambda$ if and only if it obeys
the identity $u=(uu^\lambda )\cdot u(uu^\lambda )^\rho$.
\end{mylem}
\begin{proof}
This is proved by using the identity OSI$_{01.2.3}^\rho$ of
Lemma~\ref{1:9.11a}.\end{proof}

\begin{mylem}\label{1:9.11j}
A right universal Osborn loop $(Q, \cdot ,\backslash ,/)$ obeys the
identity $u^\rho u=uu^\lambda$ if and only if it obeys the identity
$u\cdot u^\lambda u^\rho=u^\rho$.
\end{mylem}
\begin{proof}
This is proved by using the identity $u\cdot[u\backslash (u^\rho
u)]u^\rho =u^\rho$ of Lemma~\ref{1:9.11a}.\end{proof}

\begin{mylem}\label{1:9.11k}
A right universal Osborn loop $(Q, \cdot ,\backslash ,/)$ obeys the
identity $u^\rho u=uu^\lambda$ if and only if it obeys the identity
$\{(uu)\cdot u\backslash [(u^\rho u)[(uu)\backslash u]]\}u=uu\cdot
u^\lambda$.
\end{mylem}
\begin{proof}
This is proved by using the identity OSI$_{01.2.8}^\rho$ of
Lemma~\ref{1:9.11a}.\end{proof}

\begin{mycor}\label{1:9.11l}
A right universal Osborn loop $(Q, \cdot ,\backslash ,/)$ that obeys
the identity $u^\rho u=uu^\lambda$ and the RSIP is a SWIPL.
\end{mycor}
\begin{proof}
This can be deduced from Lemma~\ref{1:9.11k}.\end{proof}

\section{Concluding Remarks and Future Studies}
Identities OSI$_{01}$, OSI$_{01.\ldots}$; OSI$_{01}^\rho$,
OSI$_{01.\ldots}^\rho$ and OSI$_{01}^\lambda$,
OSI$_{01.\ldots}^\lambda$ are all newly discovered identities that
are true in universal, right universal and left universal Osborn
loops respectively. So they are all obeyed by any Moufang loop,
extra loop, CC-loop, universal WIPL and VD-loop. This is a good news
for CC-loop which has just received a tremendious growth increase by
the works of Kinyon, Kunen, Drapal, Phillips e.t.c and especially
for VD-loops which is yet to grow in study compared to CC-loops. We
hope VD-loops will catch the attention of researchers with the newly
found identities. A trilling observation in this study is the fact
that identities OSI$_{01}^\lambda$ and OSI$_{01}$ are of the forms
\begin{displaymath}
[y(x^{-1}v)\cdot v^{-1}](xv)=[y(xv)\cdot
v^{-1}](x^{-1}v)~\textrm{and}
\end{displaymath}
\begin{displaymath}
y\{u^{-1}([(uv)(v^{-1}x^{-1}\cdot
u)]v)\}=\{(y[u^{-1}(xv)])v^{-1}\cdot x^{-1}(uv)\}[v^{-1}x^{-1}\cdot
u]\cdot v.
\end{displaymath}
respectively, in a Moufang loop or extra loop. If a Moufang or extra
loop is of exponent 2 then, the first identity will be obviously
true. Basarab \cite{phd46} has shown that an Osborn loop of exponent
2 is an abelian group. So it is not wise to study identity
OSI$_{01}^\lambda$ for a loop of exponent 2 e.g. Steiner loops, but
identity OSI$_{01}$ can be studied for such a loop.

According to Phillips \cite{phd151}, a chain of five prominent
varieties of CC-loops are: (1) groups, (2) extra loops, (3) WIP
PACC-loops, (4) PACC-loops and (5) CC-loops. He was able to
axiomatize the variety of WIP PACC-loops. With our new loop
properties that are weaker forms of well known loop properties like
inverse property, power associativity and diassociativity, we now
have subvarieties of varieties of CC-loops mentioned above. It will
be interesting to axiomatize some of them e.g. SWIP PACC-loops.
These new algebraic properties give more insight into the algebraic
properties of universal Osborn loops. Particularly, it can be used
to fine tune some recent equations on CC-loop as shown in works of
Kunen, Kinyon, Phillips and Drapal; \cite{phd35,phd36,phd47},
\cite{phd37,phd38},
 \cite{phd78}.

The continuation of this study will switch to the notations of
Bryant and Schneider \cite{phd92} for principal isotopes of
quasigroups (loops) and use their results to deduce more algebraic
equations for universal Osborn loops.

~\\
\begin{enumerate}
\item Department of Mathematics,\\
Obafemi Awolowo University,\\ Il\'e If\`e, Nigeria.\\
{\bf e-mail}: jaiyeolatemitope@yahoo.com, tjayeola@oauife.edu.ng  \item Department of Mathematics,\\
University of Agriculture,\\ Ab\d e\'ok\`uta 110101, Nigeria.\\
{\bf e-mail}: ekenedilichineke@yahoo.com, adeniranoj@unaab.edu.ng
\end{enumerate}

\begin{thebibliography}{99}
\bibitem{phd30} R. Artzy: {\it Crossed inverse and related
loops}, Trans. Amer. Math. Soc. \textbf{91}(2) (1959), 480--492.
\bibitem{phd149} A. S. Basarab: {\it A class of WIP-loops},
Mat. Issled. \textbf{2}(2) (1967), 3--24.
\bibitem{phd148} A. S. Basarab: {\it The Osborn loop}, Studies in the theory of quasigroups and
loops, Shtiintsa, Kishinev \textbf{193} (1973), 12--18.
\bibitem{phd146} A. S. Basarab: {\it Isotopy of WIP
loops}, Mat. Issled. \textbf{5}2(16) (1970), 3-12.
\bibitem{phd46} A. S. Basarab: {\it Osborn's G-loop}, Quasigroups and Related
Systems \textbf{1}(1) (1994), 51--56.
\bibitem{phd137} A. S. Basarab: {\it Generalised Moufang G-loops}, Quasigroups and Related
Systems \textbf{3} (1996), 1--6.
\bibitem{phd170} A. S. Basarab and A.
I. Belioglo: {\it UAI Osborn loops}, Quasigroups and loops, Mat.
Issled. \textbf{51} (1979), 8--16.
\bibitem{phd41} R. H. Bruck: {\it A survey of binary systems}, Springer-Verlag, Berlin-G\"ottingen-Heidelberg, 1966.
\bibitem{phd92} B. F. Bryant and H. Schneider: {\it Principal
loop-isotopes of quasigroups}, Canad. J. Math. \textbf{18} (1966),
120--125.
\bibitem{phd39} O. Chein, H. O. Pflugfelder and J. D. H. Smith: {\it Quasigroups and loops : Theory and applications}, Heldermann Verlag, 1990.
\bibitem{phd96} V. O. Chiboka: {\it The study of properties and
construction of certain finite order G-loops}, Ph.D thesis, Obafemi
Awolowo University, Ile-Ife, 1990.
\bibitem{phd49} J. Dene and A. D. Keedwell: {\it Latin
squares and their applications}, the English University press Lts,
1974.
\bibitem{phd37} A. Dr\'apal: {\it Conjugacy closed loops and
their multiplication groups}, J. Alg. \textbf{272} (2004), 838--850.
\bibitem{phd38} A. Dr\'apal: {\it Structural interactions of conjugacy closed
loops}, Trans. Amer. Math. Soc. \textbf{360} (2008), 671--689.
\bibitem{phd159} E. Falconer: {\it Quasigroup identities invariant under
isotopy}, Ph.D thesis, Emory University, 1969.
\bibitem{phd160} E. Falconer: {\it Isotopy invariants in
quasigroups}, Trans. Amer. Math. Soc. \textbf{151}(2) (1970),
511--526.
\bibitem{phd50} F. Fenyves: {\it Extra loops I}, Publ. Math. Debrecen, \textbf{15} (1968), 235--238.
\bibitem{phd56} F. Fenyves: {\it Extra loops II}, Publ. Math. Debrecen, \textbf{16} (1969), 187--192.
\bibitem{phd42} E. G. Goodaire, E. Jespers and C. P. Milies: {\it Alternative loop rings}, NHMS(184), Elsevier, 1996.
\bibitem{phd44} E. D. Huthnance Jr.: {\it A theory of
generalised Moufang loops}, Ph.D. thesis, Georgia Institute of
Technology, 1968.
\bibitem{phd33} M. K. Kinyon: {\it A survey of Osborn loops},
Milehigh conference on loops, quasigroups and non-associative
systems, University of Denver, Denver, Colorado, 2005.
\bibitem{phd36} M. K. Kinyon, K. Kunen: {\it The structure of
extra loops}, Quasigroups and Related Systems \textbf{12} (2004),
39--60.
\bibitem{phd47} M. K. Kinyon, K. Kunen: {\it
Power-associative conjugacy closed loops}, J. Alg. \textbf{304}(2)
(2006), 679--711.
\bibitem{phd35} M. K. Kinyon, K. Kunen, J. D. Phillips: {\it
Diassociativity in conjugacy closed loops}, Comm. Alg. \textbf{32}
(2004), 767--786.
\bibitem{phd124} M. K. Kinyon, J. D. Phillips and P. Vojt\v echovsk\'y: {\it Loops of Bol-Moufang type with a subgroup of index
two}, Bul. Acad. Stiinte Repub. Mold. Mat. \textbf{49}(3) (2005),
71--87.
\bibitem{phd78} K. Kunen: {\it The structure of conjugacy closed loops}, Trans. Amer. Math. Soc. \textbf{352} (2000), 2889--2911.
\bibitem{phd185} K. Kunen: {\it G-loops and Permutation Groups}, J. Alg. \textbf{220} (1999), 694--708.
\bibitem{phd88} P. T. Nagy and K. Strambach: {\it Loops as
invariant sections in groups, and their geometry}, Canad. J. Math.
\textbf{46}(5) (1994), 1027--1056.
\bibitem{phd89} J. M. Osborn: {\it Loops with the weak
inverse property}, Pac. J. Math. \textbf{10} (1961), 295--304.
\bibitem{phd3} H. O. Pflugfelder: {\it Quasigroups and loops : Introduction}, Sigma series in Pure Math. 7, Heldermann Verlag, Berlin, 1990.
\bibitem{phd151} J. D. Phillips: {\it A short basis for the variety of WIP
PACC-loops}, Quasigroups and Related Systems \textbf{14}(1) (2006),
73--80.
\bibitem{phd9} J. D. Phillips and P. Vojt\v echovsk\'y: {\it The varieties of loops of Bol-Moufang type}, Alg. Univer. \textbf{54}(3) (2005), 259--383.
\bibitem{phd61} J. D. Phillips and P. Vojt\v echovsk\'y: {\it The varieties of quasigroups of Bol-Moufang type : An equational
approach}, J. Alg. \textbf{293} (2005), 17--33.
\bibitem{phd75} W. B. Vasantha Kandasamy: {\it Smarandache
loops}, Department of Mathematics, Indian Institute of Technology,
Madras, India, 2002.
\end{thebibliography}
\end{document}